\documentclass[a4paper]{amsart}

\usepackage[english]{babel}
\usepackage[utf8]{inputenc}
\usepackage{amsmath}
\usepackage{amssymb}
\usepackage{amsthm}
\usepackage{MnSymbol}
\usepackage{tikz-cd}
\usepackage{graphicx}
\usepackage{mathrsfs}
\usepackage[colorinlistoftodos]{todonotes}

\title{Discriminants and Quasi-symmetry}
\author{Alex Kite \vspace{-10ex}}

\theoremstyle{plain}
\newtheorem{thm}{Theorem}[section]
\newtheorem{lem}[thm]{Lemma}

\newtheorem*{thm*}{Theorem}
\newtheorem*{lem*}{Lemma}
\newtheorem*{conj}{Conjecture}
\newtheorem*{cor*}{Corollary}
\newtheorem*{prop*}{Proposition}
\newtheorem{prop}[thm]{Proposition}

\theoremstyle{definition}
\newtheorem{df}[thm]{Definition}
\newtheorem*{df*}{Definition}

\newtheorem{df/thm}[thm]{Definition/Theorem}

\theoremstyle{remark}
\newtheorem{rem}[thm]{Remark}
\newtheorem*{ack}{Acknowledgements}

\DeclareMathOperator{\im}{Im}

\begin{document}
\maketitle

\begin{abstract}
This paper gives a geometric interpretation of the notion of quasi-symmetric representation and uses this to show that the discriminant locus associated to such a representation is a hyperplane arrangement. Moreover, we identify this hyperplane arrangement, up to a shift, with the one appearing in recent work of Halpern-Leistner--Sam 
\end{abstract}

\section{Introduction}

Start with a representation $T_L \lcirclearrowright \mathbb{C}^n$ of a connected torus $T_L$ over $\mathbb{C}$. The weights of this representation $\beta_1,\ldots,\beta_n$ lie in $L^\vee$, where $L$ is the lattice of cocharacters of $T_L$. A torus representation can be viewed geometrically through the associated \textit{toric variation of GIT} (VGIT). From this perspective, a choice of $\beta \in L^{\vee}$ determines an equivariant polarisation $\mathcal{O}(\beta)$ on $\mathbb{C}^n$ and we can form the associated  GIT quotient $X_{\beta}:=\mathbb{C}^n//_{\beta} T_L$, which is a toric variety. Our representation is \emph{Calabi--Yau} if the sum of the weights is 0 and if this holds, then $X_{\beta}$ is a \emph{Calabi--Yau variety} -- that is, $K_{X_{\beta}}\cong \mathcal{O}_{X_{\beta}}$. The space of parameters for the VGIT is therefore $L^{\vee}_{\mathbb{R}}$ and this comes with a wall and chamber decomposition \cite{DH}. Any two $\beta$ lying in the interior of a chamber give isomorphic varieties $X_{\beta}$ and crossing a wall corresponds to a birational modification of $X_{\beta}$. We call $X_{\beta}$ a \textit{phase} of the VGIT. 

In general, birational modifications of Calabi--Yaus are expected to induce equivalences between their derived categories and this has been proved for toric VGITs using the theory of windows in \cite{BFK} and \cite{HL}. However such equivalences are not unique and so we would like to understand the global story of how they all fit together. A priori it seems difficult to guess what relations one expects between these equivalences.

Yet physicists came up with a remarkable prediction, which we shall now explain. They first tell us to complexify the wall and chamber structure on the space of parameters. We do this by associating to our VGIT a particular type of 2-dimensional quantum field theory, called a \textit{gauged linear sigma model} (GLSM). Our VGIT parameters $\beta \in L^{\vee}_{\mathbb{R}}$ give parameters in this theory. These should be thought of as K\"ahler parameters -- indeed the polarisation $\mathcal{O}(\beta)$ descends to a canonical one on $X_{\beta}$. There are other ``K\"ahler-type" parameters in this theory, analogous to $B$-fields on a Calabi--Yau variety. Together these parameters form a complex manifold called the \emph{Fayet-Iliopoulos parameter space} (FIPS), which is a version of the ``stringy K\"ahler moduli space" of a Calabi--Yau variety. In general, it is hard to make global mathematical sense of the FIPS, though it is closely related to the space of Bridgeland stability conditions \cite{B}. However in our situation, toric mirror symmetry allows us to identify the FIPS with the ``complex parameter space" of the mirror GLSM. Following  Gelfand, Kapranov and Zelevinsky \cite{GKZ}, this is the complement of some ``GKZ discriminant locus" in the dual torus $T_{L^{\vee}}=\text{Hom}(L,\mathbb{C}^*)$, and it is this discriminant locus which acts as the complexification of our walls.

So we have a family of GLSMs parameterised by the FIPS. In these physical theories there are dynamical objects known as ``D-branes" which form a category. This category can be identified with the derived category of the associated phase when we choose our FI parameters to have $\beta$ lying sufficiently deep inside a chamber. Starting with a ``D-brane" in a particular phase, physicists claim they can canonically transport it to different phases as we vary the FI parameters. However mathematicians have yet to make ``D-brane transport", or indeed the GLSM itself, rigorous. Nonetheless the existence of this ``local system of categories" over the FIPS is a testable mathematical statement. Fixing a base point at a phase leads to the following conjecture:

\begin{conj}
There is an action of $\pi_1(\text{FIPS})$ on $D^b(X_{\beta})$ for any phase $X_{\beta}$.
\end{conj}

Several people have constructed actions along these lines. Of particular relevance to us are Donovan--Segal's \cite{DS} examples arising from $A_n$ surface singularities. In \cite{DS}, the authors explicitly identify the FIPS with a hyperplane complement and show that its fundamental group acts on the derived category.  Recently Halpern-Leistner--Sam~\cite{HLS} have considered VGITs arising from so-called quasi-symmetric representations (generalising \cite{DS}), using work of \v{S}penko--Van den Bergh~\cite{SV} on non-commutative crepant resolutions. In this setting, they construct an action of the fundamental group of a certain hyperplane complement on the phases of this VGIT. However they leave open the question of whether this hyperplane complement is indeed the FIPS. In this note, we answer this question affirmatively.

\begin{prop*}
The hyperplane arrangement constructed by Halpern-Leistner--Sam in \cite{HLS} agrees with the GKZ discriminant locus up to an overall shift. 
\end{prop*}

\noindent As this shift doesn't affect the topology of the FIPS, we see that the physicists' conjecture is true for all quasi-symmetric toric VGITs.

Actions of hyperplane complements have also appeared in the context of geometric representation theory. For example, Bezrukavnikov--Riche \cite{BR} study examples coming from the Grothendieck-Springer resolution of a semi-simple Lie algebra $\mathfrak{g}$ (this covers the examples in \cite{DS}) and construct actions of the fundamental group of the complement in $\mathfrak{h}^{\vee}$ of the reflection hyperplanes of the Weyl group of $\mathfrak{g}$.
Donovan--Wemyss \cite{DW} also construct actions of fundamental groups of hyperplane complements but it is not yet clear to us how their hyperplane complement relates to the FIPS.

The examples of Donovan--Segal and Halpern-Leistner--Sam all start with a special type of torus representation, called a \emph{quasi-symmetric} representation. 
\begin{df*}{(\cite{SV}, 1.6)} A representation $T_L \lcirclearrowright \mathbb{C}^n$ is \textit{quasi-symmetric} if and only if, for each line $\ell$ in $L^\vee_{\mathbb{R}}$, the sum of the weights $\beta_i$ that lie on $\ell$ is zero.
\end{df*}

\noindent An important example  of quasi-symmetric representations from the geometric representation theory perspective  are \textit{self-dual} representations -- that is, $\mathbb{C}^n \cong (\mathbb{C}^n)^{\vee}$ as representations. We also note that quasi-symmetric representations are necessarily \textit{Calabi--Yau} and these conditions are equivalent when $L$ has rank 1. 

Our main result answers the natural question: why does the quasi-symmetry condition simplify the construction of these categorical actions so much? 

\begin{thm*}
The discriminant locus of a quasi-symmetric representation is a hyperplane arrangement.
\end{thm*}

Conversely, we expect that it is hard to write down interesting examples of Calabi--Yau representations whose discriminant locus is a hyperplane arrangement but which are not quasi-symmetric. For example, this cannot happen with a FIPS of dimension 1 or 2. 

Our theorem follows in a straightforward fashion from the following lemma. As explained in \S \ref{sec:discr}, the discriminant locus naturally breaks up into various ``discriminant varieties" of different dimensions. In this collection, there is a ``primary" one, denoted $\nabla_{pr}$.  

\begin{lem*}
A torus representation $T_L \lcirclearrowright \mathbb{C}^n$ is quasi-symmetric if and only if it is Calabi--Yau and $\nabla_{pr}$ is a point.
\end{lem*}
 
We hope that our results will motivate interest in the non-quasi-symmetric case where the geometry of the discriminant is richer and the combinatorics of toric VGITs is harder.

In \S\ref{sec:discr}, we recall the theory of GKZ discriminants and determinants in the toric setting from \cite{GKZ}. Proofs of the Theorem and Lemma above are the content of \S\ref{sec:main}. In \S\ref{sec:HLS}, we introduce the hyperplane arrangement in \cite{HLS} and prove it agrees with the GKZ discriminant locus.

\begin{ack}
I am indebted to Ed Segal for raising this question and to both him and Tom Coates for numerous useful discussions.
\end{ack}

\noindent \emph{Notation:}
Let $e_i$ denote the $i$th standard basis vector for $\mathbb{Z}^n$ and $e_i^\vee$ denote the $i$th standard basis vector for $(\mathbb{Z}^n)^\vee$.  The weights give a map $Q: (\mathbb{Z}^n)^{\vee} \rightarrow L^{\vee}$ defined by $Q(e_i^{\vee}):=\beta_i$.  Without loss of generality we may assume that $Q$ is surjective, as replacing $L^\vee$ by $\im(Q)$ leaves both the discriminant and the notion of quasi-symmetry unchanged. Set $k:=\text{rank}(L)$.  Let $M:=\text{ker}(Q)$ and $N:=M^{\vee}$; these are both lattices of rank $n-k$. There are exact sequences:
\begin{align}
0 & \rightarrow L  \xrightarrow{Q^{\vee}} \mathbb{Z}^n  \xrightarrow{A} N   \rightarrow 0 \nonumber \\
0 & \rightarrow M  \xrightarrow{A^{\vee}} (\mathbb{Z}^n)^{\vee}  \xrightarrow{Q} L^{\vee}  \rightarrow 0 
\label{ToricData}
\end{align} 
The map $A$ here is called the \emph{ray map} and $\omega_i:=A(e_i)$ is the $i$th \textit{ray}. We assume that the rays are distinct. From the perspective of toric VGIT, we usually think of $A$ as the starting data; cf.~\cite{DH}. Let $\sigma \subset N_{\mathbb{R}}$ be the cone generated by the rays. For $l \in L$, set $H_l:=\{ y \in L_{\mathbb{R}}^{\vee} \mid \langle l, y \rangle = 0\}$.

\section{Discriminants in the toric setting}
\label{sec:discr}

We now introduce the principal A-determinant and Horn uniformization following \cite{GKZ}, Ch. 9 and 10. Abusing notation, let $A \subset N$ be the set of rays. Henceforth we consider only torus representations that are Calabi--Yau. This allows us to find an affine hyperplane $H$, of the form $\langle m,- \rangle =1$ for some $m \in M$, on which all the rays lie. So we can equivalently think of $A$ as a subset of the polytope $\Delta:=\sigma \cap H$. 

Identifying elements $\omega \in N=M^{\vee}$ with characters $x^{\omega}$ of $T_M$, we may consider $\mathbb{C}^{A}:=\{f(x)=\sum_{\omega \in A} a_{\omega} x^{\omega} \}$, the set of functions on $T_M$ with exponents in $A$. 
\begin{df}{(\cite{GKZ}, Ch. 9, 1.2)} Set
\[
\nabla_0:=\{ f(x)=\{a_{\omega}\}_{\omega \in A} \in \mathbb{C}^A \mid \text{$f$ has a critical point in $T_M$} \}
\]
and $\nabla_A:=\bar{\nabla}_0 \subset \mathbb{C}^A$. When $\nabla_A$ is a hypersurface, its defining equation is called the \textit{A-discriminant} $\Delta_A(\{a_{\omega}\}_{\omega \in A})$. Otherwise we declare $\Delta_A=1$.
\label{A-disc}
\end{df}

Gelfand, Kapranov, and Zelevinsky give an intrinsic definition, along the lines of Definition~\ref{A-disc}, of the  \textit{principal A-determinant} $E_A \in \mathbb{C}[\{a_\omega\}]$, and, in \cite{GKZ}, Ch. 10, Thm 1.2, show that (up to a sign) 
\begin{equation}
  \label{eq:E_A}
  E_A(\{a_{\omega}\}_{\omega \in A}):= \prod_{\Gamma \subset \Delta} \Delta_{A \cap \Gamma}^{m(\Gamma)}.
\end{equation}
Here the multiplicity $m_\Gamma$ (defined in \cite{GKZ}, Ch. 10, 1.B) is a non-negative integer, and we interpret $\Delta_{A \cap \Gamma}$ as a function on $\mathbb{C}^A$ by pulling-back under the natural projection $\mathbb{C}^A \twoheadrightarrow \mathbb{C}^{A \cap \Gamma}$.
A~priori, $\nabla_A$ and $\{E_A=0\}$ are defined inside $(\mathbb{Z}^n)^{\vee} \otimes \mathbb{C}$ but, in fact, they have $k$ quasi-homogeneities (\cite{GKZ}, Ch. 9, 3.B) meaning that they descend to $T_{L^{\vee}}$. When we write $\nabla_A$ or $\{E_A=0\}$, we shall mean this version of the A-discriminant/A-determinant. 

\begin{df}
The \textit{discriminant locus} is the subset $\{E_A=0\} \subset T_{L^{\vee}}$. The \textit{FIPS} associated to $A$ is the complement of the discriminant locus inside $T_{L^{\vee}}$.
\label{FIPS} 
\end{df}

\begin{rem}
One might expect the (reduced) discriminant locus  to be $\bigcup_{\Gamma \subset \Delta} \nabla_{A \cap \Gamma}$. It is not clear to us whether this agrees with our definition.
\end{rem}

\begin{rem}
The universal cover $\pi: L^{\vee} \otimes \mathbb{C} \rightarrow T_{L^{\vee}}$ induces a covering $\widetilde{\text{FIPS}} \rightarrow$ FIPS by restricting along the inclusion FIPS $\subset T_{L^{\vee}}$. Since $\pi$ is given in coordinates by taking the logarithm, we prefix the pullback under this cover of any object defined on $T_{L^{\vee}}$ by \emph{log}. For example, the \emph{log}-discriminant locus is the union of $\{\text{Log}(\Delta_{A \cap \Gamma})=0\}$ over all faces $\Gamma \subset \Delta$, where $\text{Log}:=\frac{\text{log}}{2\pi i}$. As $\text{Log}$ is multivalued, $\{\text{Log}(\Delta_{A \cap \Gamma})=0\}$ consists of translates under $L^{\vee}$ of $\{\text{Log}_{br}(\Delta_{A \cap \Gamma})=0\}$, where $\text{Log}_{br}$ is the single-valued version of $\text{Log}$ with arguments lying in $[0,2\pi)$. 
\end{rem}

\begin{rem}
Recall (see e.g \cite{CLS} Ch. 14 or \cite{DH}) that  the weights $\beta_i$ determine a fan in $L^{\vee}_{\mathbb{R}}$, called the ``secondary fan". The maximal cones of this fan parametrise all the different quasi-projective simplicial fans with support $\sigma$ and whose rays form a subset of the original rays in $N$. It is helpful to think of the discriminant locus as a ``detropicalisation" of the codimension-1 cones in the secondary fan -- that is, the latter are the asymptotic directions of the discriminant locus. The precise statement is that the secondary fan is the normal fan of the Newton polytope of $E_A$ (\cite{GKZ}, Ch. 10, Thm 1.4).
\end{rem}

\begin{df}
The \textit{primary discriminant} of the VGIT is $\nabla_{pr}:=\nabla_A \subset T_{L^{\vee}}$. When it is a hypersurface, we call it the \textit{primary component}. 
\label{Pr}
\end{df}

\begin{df/thm}{(\cite{GKZ}, Ch. 9, 3.C)}
The \textit{Horn uniformization} is the rational map with image $\nabla_{pr}$ given by:
\begin{align*}
\mathbb{P}(L \otimes \mathbb{C}) & \dashedrightarrow \nabla_{pr} \subset T_{L^{\vee}}=\text{Hom}(L,\mathbb{C}^*) \\
{[a_1,\hdots ,a_n]} & \mapsto ((b_1, \hdots , b_n) \mapsto \prod_{i=1}^n a_i^{b_i} )
\end{align*} 
where we identify $L$ with its image inside $\mathbb{Z}^n$. In the case when $\nabla_{pr}$ is a hypersurface, this is a \textit{birational} map.
\end{df/thm}

If we pick a basis for $L$ and corresponding coordinates $\lambda_1, \hdots, \lambda_k$ on $L \otimes \mathbb{C}$, then (identifying $T_{L^{\vee}} \cong (\mathbb{C}^*)^k$) we may rewrite the Horn uniformization as:

\begin{align}
\mathbb{P}^{k-1} & \dashedrightarrow \nabla_{pr} \subset (\mathbb{C}^*)^k \nonumber \\ 
{[\lambda_1, \hdots, \lambda_k]} & \mapsto ( \prod_{j=1}^n ( \lambda_1 \beta_{j1}+ \hdots + \lambda_k \beta_{jk} )^{\beta_{ji}}  )_{i=1, \hdots, k} \label{Horn}
\end{align}
where $\beta_{ji}$ are the components of $\beta_j \in L^{\vee}\cong \mathbb{Z}^k$ in the dual basis.

\section{Quasi-symmetry and discriminants}
\label{sec:main}

The proof of Theorem \ref{Cor} is based on the following observation:
\begin{lem}
The representation $T_L \lcirclearrowright \mathbb{C}^n$ is quasi-symmetric if and only if $\nabla_{pr}$ is a point.
\label{Main}
\end{lem}

\begin{proof}
$\nabla_{pr}$ is a point if and only if its Horn uniformization is constant. From \eqref{Horn}, this happens precisely when, for all $i$, $\prod_{j=1}^n ( \lambda_1 \beta_{j1}+ \hdots + \lambda_k \beta_{jk} )^{\beta_{ji}} $
is constant as a degree 0 element of $\mathbb{C}(\lambda_1, \hdots, \lambda_k)$.
Since $\sum_{m=1}^k \lambda_m \beta_{jm}$ cancels with $\sum_{m=1}^k \lambda_m \beta_{Jm}$ if and only if $\beta_j$ and $\beta_{J}$ lie on the same line in $L^{\vee}_{\mathbb{R}}$, decomposing $\prod_{j=1}^n ( \lambda_1 \beta_{j1}+ \hdots + \lambda_k \beta_{jk} )^{\beta_{ji}}$ as $\prod_{\ell \subset L^{\vee}_{\mathbb{R}}} (\prod_{j| \beta_j \in \ell} (\sum_{m=1}^k \lambda_m \beta_{jm})^{\beta_{ji}}$ shows that this is constant if and only if each factor $\prod_{j| \beta_j \in \ell} (\sum_{m=1}^k \lambda_m \beta_{jm})^{\beta_{ji}}$ is constant for all $i$ and lines $\ell$. Fix a primitive generator $\underline{\ell}=(l_1, \hdots, l_k)$ for $\ell$ and write each $\beta_j$ on $\ell$ as $n_j \underline{\ell}$. Then
\begin{gather*}
\prod_{j| \beta_j \in \ell} (\sum_{m=1}^k \lambda_m \beta_{jm})^{\beta_{ji}}=(\prod_{j| \beta_j \in \ell} 
n_j^{\beta_{ji}})(\sum_{m=1}^k \lambda_m l_{m})^{\sum_{j| \beta_j \in \ell} \beta_{ji}}  
\end{gather*}
is constant if and only if $\sum_{j| \beta_j \in \ell} \beta_{ji}=0$. Hence the result.
\end{proof}

\begin{df}{(\cite{GKZ}, Ch. 7, 1.B)}
A collection of rays $\{\omega_i\} \subset N_{\mathbb{R}}$ forms a \textit{circuit} if there is precisely one linear relation between them. A face $\Gamma \subset \Delta$ is called a \textit{circuit} if the collection of all rays lying in $\Gamma$ forms a circuit.
\end{df}

If $\Gamma$ is a circuit, then the proof of Lemma \ref{Main} implies $\nabla_{A \cap \Gamma}$ is a hyperplane (it's a point), given by $\{x=c_{\Gamma}\} \subset T_{L_{\Gamma}^{\vee}}$ where $c_{\Gamma} = \prod_{j=1}^n n_j^{n_j}$ and $x$ is the coordinate on $T_{L^{\vee}}$ corresponding to a choice of generator  $l_{\Gamma} \in L_{\Gamma}$. When we pull $\nabla_{A \cap \Gamma}$  back to $T_{L^{\vee}}$ under the map induced by $p: L^{\vee} \twoheadrightarrow L_{\Gamma}^{\vee}$, it becomes $h_{\Gamma}:=\{x^{l_{\Gamma}}=c_{\Gamma}\}$, where $c_{\Gamma}$ only depends on $l_{\Gamma}$ since $n_j=\langle l_{\Gamma}, \beta_j \rangle$. In the quasi-symmetric case, if  $\underline{\ell}$ is a primitive generator of $\ell$ and $\beta_j=n_j \underline{\ell}$ we have that: 
\begin{equation}
c_{\Gamma}= \prod_{\ell \subset L_{\mathbb{R}}^{\vee}} \prod_{j \mid \beta_j \in \ell} (n_j \langle l_{\Gamma},\underline{\ell} \rangle)^{ \langle l_{\Gamma},\beta_j \rangle} = \prod_{\ell \subset L_{\mathbb{R}}^{\vee}} \prod_{j \mid \beta_j \in \ell} n_j^{\langle  l_{\Gamma}, \beta_j \rangle  }= \prod_j n_j^{\langle l_{\Gamma},\beta_j \rangle} 
\label{c}
\end{equation}
 where, for the second equality, we use that 
\[ \prod_{j \mid \beta_j \in \ell} (n_j \langle l_{\Gamma}, \underline{\ell} \rangle)^{ \langle l_{\Gamma},\beta_j \rangle} = (\prod_{j \mid \beta_j \in \ell} n_j^{\langle l_{\Gamma},\beta_j \rangle}) \langle l_{\Gamma},\underline{\ell} \rangle^{\sum_j \langle l_{\Gamma},\beta_j \rangle}=\prod_{j \mid \beta_j \in \ell} n_j^{\langle l_{\Gamma} , \beta_j \rangle} \]
and that $\sum_{j \mid \beta_j \in \ell} \beta_j =0$, by quasi-symmetry. 

Then $h_{\Gamma}$ is a log-hyperplane -- if we pick a basis of $L$ and corresponding coordinates $x_i$ on $T_{L^{\vee}}$, then $x^{l_{\Gamma}}=\prod_i x_i^{l_i}$ and $H_{\Gamma}:=\text{Log}(h)=\{ (\text{Log}(x_i))_i \in \mathbb{C}^k \mid \sum_i l_i \text{Log}(x_i) \in \text{Log}(c_{\Gamma})+\mathbb{Z} \}$ is a free $\mathbb{Z}$-orbit of complex affine hyperplanes. So circuits give rise to log-hyperplanes in the discriminant locus. 

\begin{rem}
In the self-dual case, we may pick half of the weights, which we index $\beta_i$, such that the weights $\beta_j$ are precisely those of the form $\pm \beta_i$. Then the terms in $c_{\Gamma}$ corresponding to $\pm \beta_i$ cancel up to a sign and we get that $c_{\Gamma}= \prod_i (-1)^{\langle l_{\Gamma},\beta_i \rangle}$. This means that $\Im(\text{Log}(c_{\Gamma}))=0$ and hence that $H_{\Gamma}$ is the complexification of a real hyperplane. This is not true for general quasi-symmetric representations.
\label{sd}
\end{rem}

 Our main theorem says that, in the quasi-symmetric case, all components of the discriminant locus arise from circuits. 

\begin{thm}
The log-discriminant locus associated to a quasi-symmetric representation is an (affine) hyperplane arrangement, whose hyperplanes are the log-$(A \cap \Gamma)$-discriminants arising from the faces $\Gamma \subset \Delta$ which are circuits. 

\label{Cor}
\end{thm}

The inductive structure in \eqref{eq:E_A} means that, to understand the discriminant locus, we need to consider the sub-VGIT problem associated to $\Gamma \subset \Delta$. One can check that 3 copies of the short exact sequence (\ref{ToricData}) fit together in a commutative diagram with exact rows and columns, where $n_{\Gamma}$ is the number of rays in the face $\Gamma$, $N_{\Gamma} \subset N$ is the sublattice generated by $\Gamma \cap N$ and $M_{\Gamma}=N_{\Gamma}^{\vee}$:

\begin{equation}
\begin{tikzcd}
\hspace{1cm} & 0 \arrow{d}   & 0 \arrow{d} & 0 \arrow{d} &  \\
0 \arrow{r} & M_{\Gamma}'  \arrow{d} \arrow{r} & (\mathbb{Z}^{n-n_{\Gamma}})^{\vee} \arrow{d} \arrow{r}  & (L'_{\Gamma})^{\vee} \arrow{d} \arrow{r}  & 0  \\
0 \arrow{r} & M \arrow{d} \arrow{r} & (\mathbb{Z}^n)^{\vee} \arrow{d} \arrow{r} & L^{\vee} \arrow{d}{p} \arrow{r} & 0 \\
0 \arrow{r} & M_{\Gamma} \arrow{d} \arrow{r} & (\mathbb{Z}^{n_{\Gamma}})^{\vee} \arrow{d} \arrow{r} & L_{\Gamma}^{\vee} \arrow{d} \arrow{r} & 0 \\	
 & 0 & 0 & 0 &   
\end{tikzcd}
\label{CommSquare}
\end{equation}

\begin{lem}
If the representation $T_L \lcirclearrowright \mathbb{C}^n$ is quasi-symmetric and $\Gamma \subset \Delta$ is a face, then the induced representation $T_{L_{\Gamma}} \lcirclearrowright \mathbb{C}^{n_{\Gamma}}$ is quasi-symmetric also.
\label{Inherit}
\end{lem}

\begin{proof}
Fix a line $\hat{\ell} \subset (L_{\Gamma})_{\mathbb{R}}^{\vee}$ and consider $\sum_{i| \hat{\beta}_i \in \hat{\ell}\backslash \{0\}} \hat{\beta}_i$ where $\hat{\beta}_i:=p(\beta_i) \in L_{\Gamma}^{\vee}$ for $i$ such that $\omega_i \in \Gamma$. Take a term $\hat{\beta}_i$ in this sum and consider the natural lift $\beta_i \in L^{\vee}$. It defines a line $\ell \subset L^{\vee}_{\mathbb{R}}$ and so, by quasi-symmetry, $\sum_{i| \beta_i \in \ell} \beta_i=0$. Since $l \nsubset (L'_{\Gamma})_{\mathbb{R}}^{\vee}$, the commutative diagram \eqref{CommSquare} implies that all of the $\beta_i$ in this sum correspond to rays in $\Gamma$ and hence project under $p$ to give some of the remaining $\hat{\beta}_i$ in our sum. So we've proved the sub-sum $\sum_{i | \beta_i \in \ell} \hat{\beta}_i=0$. Iterating this procedure yields the desired conclusion.
\end{proof}

\begin{proof}[Proof of Theorem \ref{Cor}]
For the first claim, we only need to show that if $\Gamma$ is not a circuit, then $\Delta_{A \cap \Gamma}$ doesn't contribute to $E_A$. In this case, the space $L_{\Gamma}$ of relations in $\Gamma$ has $\text{dim}(L_{\Gamma})>1$, hence, by Lemma \ref{Main} and Lemma \ref{Inherit}, $\nabla_{A \cap \Gamma}$ has codimension at least 2 and so $\Delta_{A \cap \Gamma}$ doesn't contribute to $E_A$.     
\end{proof}

\begin{rem}
We note here that the Horn uniformization \eqref{Horn} also gives us a criterion for when the component $\nabla_{A \cap \Gamma}$ of the discriminant locus is a log-hyperplane of the form $\{\prod_{i=1}^k x_i^{l_i}=c\}$. Namely this happens precisely when 
\[ \prod_{i=1}^k  (\prod_{j=1}^n (\lambda_1 \beta_{j1}+ \hdots + \lambda_k \beta_{jk})^{\beta_{ji}})^{l_i}= \prod_{j} (\lambda_1 \beta_{j1}+ \hdots + \lambda_k \beta_{jk})^{\langle l, \beta_{j} \rangle}\]
is constant. Decomposing this product into lines $\ell \subset L^{\vee}_{\mathbb{R}}$ as in the proof of Lemma \ref{Main}, we see that this is equivalent to $\langle l, \sum_{j \mid \beta_j \in \ell} \beta_j \rangle=0$ for all lines $\ell$ -- that is, $\sum_{j \mid \beta_j \in \ell} \beta_j=0$ for all lines $\ell \not \subset H_l$. 
\end{rem}

\section{Halpern-Leistner--Sam's Hyperplane Arrangement}
\label{sec:HLS}
 
We conclude by showing that our discriminant hyperplane arrangement in $L^{\vee}_{\mathbb{C}}$ agrees with the one constructed by Halpern-Leistner--Sam \cite{HLS}, Ch. 3.  From Theorem \ref{Cor}, we know that our hyperplane arrangement, denoted $\mathcal{H}_{disc}$, comes from faces $\Gamma \subset \Delta$ which are circuits. Explicitly, a circuit has a unique (up to sign) choice of generating relation  $l_{\Gamma} \in L_{\Gamma}$ and our hyperplanes are $H_{\Gamma,n}:=\{y \in L^{\vee}_{\mathbb{C}} \mid \langle l_{\Gamma},y \rangle = \text{Log}_{br}(c_{\Gamma})+n \}$ where $n \in \mathbb{Z}$ and $c_{\Gamma}$ is defined by \eqref{c}. Note that $H_{l_{\Gamma}}=(L'_{\Gamma})^{\vee}_{\mathbb{R}}$ and so $H_{\Gamma,n}$ is a translate of $(L'_{\Gamma})^{\vee}_{\mathbb{C}}$.

In order to define their hyperplane arrangement, Halpern-Leistner--Sam introduce two polytopes in $L^{\vee}_{\mathbb{R}}$ associated to a $T_L$-representation:
\begin{align*}
\bar{\Sigma}:=\big\{ \sum_j a_j \beta_j \mid a_j \in [-1,0] \big\} =: \sum_j [-\beta_j,0] \\
\bar{\nabla}:=\big \{ \beta \in L^{\vee}_{\mathbb{R}} \mid -\frac{\eta_l}{2} \leq \langle l, \beta \rangle \leq \frac{\eta_l}{2} \text{ for all } l \in L \big \} 
\end{align*}
where $\eta_l:=\text{max}\{\langle l,\mu \rangle \mid \mu \in \bar{\Sigma} \}$. Moreover they show, in \cite{HLS}, Lemma 2.8, that for quasi-symmetric torus representations $\bar{\nabla}=\frac{1}{2}\bar{\Sigma}$. They then define a real hyperplane arrangement in $L^{\vee}_{\mathbb{R}}$ as the $L^{\vee}$ orbit, acting by translations, of the supporting affine hyperplanes $H_{F}$ of the facets $F$ of $\bar{\nabla}$. Their hyperplane arrangement, which we denote $\mathcal{H}_{HLS}$, is then the complexification of this real hyperplane arrangement.  Note that, as the weights span $L^{\vee}$, $\bar{\Sigma}$ and $\bar{\nabla}$ are full-dimensional and so there is a unique $H_F$ for each facet $F$. Therefore if we write $H_F=\{y \in L^{\vee}_{\mathbb{R}} \mid \langle l_F,y\rangle=c_F\}$, then the hyperplanes in $\mathcal{H}_{HLS}$ have the form $H_{F,n}:=\{y \in L^{\vee}_{\mathbb{C}} \mid \langle l_F,y\rangle=c_F+n\}$.

These two hyperplane arrangements, $\mathcal{H}_{disc}$ and $\mathcal{H}_{HLS}$, cannot be precisely the same since $\text{Log}_{br}(c_{\Gamma})$ can have non-zero imaginary part (though not in the self-dual case -- see Remark \ref{sd}) and hence $H_{\Gamma,n}$ cannot in general be the complexification of a real hyperplane in $L^{\vee}_{\mathbb{R}}$. However we do have:

\begin{prop}
$\mathcal{H}_{disc}=\mathcal{H}_{HLS}$ after translating the latter by $-\frac{i}{2\pi} \sum_{j} \text{log}(|n_j|)\beta_j \in iL_{\mathbb{R}}^{\vee}$
\end{prop}

\begin{proof}
We start by determining the facets of $\bar{\nabla}$ using its description both in terms of inequalities and as a convex hull. By adding up all the weights on a ray inside a given line, we can assume  that, for the purpose of determining the facets, on any line with weights our representation only has weights $\pm \beta_{j}$.

We observe that $\beta \in \partial \bar{\nabla}$ precisely when $\beta = \sum_j a_j {\beta}_j \in \bar{\nabla}=\frac{1}{2} \bar{\Sigma}$ (i.e. $a_j \in [-1/2,1/2]$) and $\exists l \in L$ such that one of inequalities in the definition of $\bar{\nabla}$ is saturated i.e. $\langle l, \beta \rangle = \pm \frac{\eta_l}{2}$. Since $\bar{\nabla}=\frac{1}{2} \bar{\Sigma}$, $\frac{\eta_l}{2}=\text{max}\{\langle l,\mu \rangle \mid \mu \in \bar{\nabla} \}$ and, as $\bar{\nabla}=-\bar{\nabla}$, $-\frac{\eta_l}{2}=\text{min}\{\langle l,\mu \rangle \mid \mu \in \bar{\nabla} \}$. Write $\langle l,\beta \rangle=\sum_j a_j \langle  l, \beta_j \rangle$ and, swapping $\beta_j$ for $-\beta_j$ if necessary, suppose that $\langle l, \beta_j \rangle \geq 0$ for all $j$. Then we have the equality $\langle l, \beta \rangle = \pm \frac{\eta_l}{2}$ precisely when $\beta$ is of the form $ \sum_{j \mid \langle l, \beta_j \rangle > 0} \pm \beta_j/2+ \sum_{j \mid \langle l, \beta_j \rangle = 0} a_j \beta_j$. So we get a pair of facets $F_{\pm}$ of $\bar{\nabla}$, going through $ \sum_{j \mid \langle l, \beta_j \rangle > 0} \pm \beta_j/2$ respectively, precisely when the set of weights in the linear hyperplane $H_l$ span the whole hyperplane. Hence the supporting hyperplanes of the facets $F_{\pm}$ are $H_{F_{\pm}}=\{\langle l,-\rangle = \langle l,\sum_{j \mid \langle l, \beta_j \rangle > 0} \pm \beta_j/2 \rangle \} $ whenever $l$ satisfies this property. So for such an $l$, $l_{F_{\pm}}:=l$ and $c_{F_{\pm}}:=\langle l,\sum_{j \mid \langle l, \beta_j \rangle > 0} \pm \beta_j/2 \rangle$. Note that, as the facets through $\sum_{j \mid \langle l, \beta_j \rangle > 0} \pm \beta_j/2$ differ by an element of $L^{\vee}$, $H_{F_+ ,m}=H_{F_-,m+n}$ for some $n \in \mathbb{Z}$, so we need only consider one of these sets of hyperplanes. By convention, we choose to work with $F_+$.

We now show that $H_{l_{F}}=H_{l_{\Gamma}}$ for some circuit in a face $\Gamma$.  We know that $H_{l_F}$ is a linear hyperplane in $L_{\mathbb{R}}^{\vee}$ which is spanned by the weights on it. First note that all the linear hyperplanes $H_{l_{\Gamma}}$ have this property, because of the commutative diagram (\ref{CommSquare}) and the fact that $H_{l_{\Gamma}}=(L'_{\Gamma})^{\vee}_{\mathbb{R}}$. Conversely, any linear hyperplane $H_{l_F}$ spanned by the weights lying on it is necessarily of this form. To see this, notice that we get a commutative diagram of lattices of the form (\ref{CommSquare}) by defining $(L_{\Gamma}')^{\vee}:=H_{l_F} \cap L^{\vee}$ and dualising this gives a linear slice $N_{\Gamma}$ of $N$. By \cite{CLS}, Lemma 14.3.3, this linear slice is a face of $\sigma$ precisely when there is a positive relation amongst the weights in $H_{l_F}$. Since the representation is quasi-symmetric, we have that $\sum_{j \mid \beta_j \in H_{l_F}} \beta_j=0$ and hence we have such a linear relation. Thus $N_{\Gamma}$ comes from a face and, as $(L'_{\Gamma})^{\vee}$ is codimension 1, $\Gamma$ is a circuit.  Since $l_{\Gamma}$ is the defining equation of $(L'_{\Gamma})^{\vee}$, $H_{l_F}=H_{l_{\Gamma}}$ as claimed.

For this $\Gamma$, we now claim that the real hyperplanes $\Re(H_{F,0})=\{y \in L^{\vee}_{\mathbb{R}} \mid \langle l_F,y\rangle=c_F\}$ and $\Re(H_{\Gamma,n})=\{y \in L_{\mathbb{R}}^{\vee} \mid \langle l_{\Gamma},y\rangle = \Re(\text{Log}_{br}(c_{\Gamma})) +  n \}$ are the same for some $n \in \mathbb{Z}$. Hence $\Re(H_{F,m})=\Re(H_{\Gamma,m+n})$ and so $\mathcal{H}_{HLS}$ and $\mathcal{H}_{disc}$ determine the same real hyperplane arrangements. To see the claim, recall that $c_F=\langle l,\sum_{j \mid \langle l, \beta_j \rangle > 0} \beta_j/2 \rangle$ and hence is a half-integer. On the other hand, $c_{\Gamma} \in \mathbb{Q}^{\times}$ so $\Re(\text{Log}_{br}(c_{\Gamma}))$ is either $0$ or $1/2$ when $c_{\Gamma}$ is respectively positive or negative. Since $H_{l_F}=H_{l_{\Gamma}}$, it suffices to prove that $c_{\Gamma}>0$ precisely when $c_F \in \mathbb{Z}$. Recall that $c_{\Gamma}$ is a product of terms of the form $n_j^{n_j}$. These are positive precisely when $n_j \geq 0$ or $n_j$ is even. So $c_{\Gamma}$ is positive precisely when it contains an even number of negative odd terms -- that is, $\sum_{n_j <0} n_j \in 2\mathbb{Z}$. Since $n_j = \langle l_{\Gamma},\beta_j \rangle$, $c_{\Gamma}>0$ precisely when $\langle l_{\Gamma},\sum_{j \mid \langle l_{\Gamma}, \beta_j \rangle <0} \beta_j/2 \rangle \in \mathbb{Z}$. By quasi-symmetry $\sum_{j \mid \langle l_{\Gamma}, \beta_j \rangle <0} \beta_j/2=-\sum_{j \mid \langle l_{\Gamma}, \beta_j \rangle >0} \beta_j/2$ and so this happens precisely when $c_F=\langle l_{\Gamma},\sum_{j \mid \langle l_{\Gamma}, \beta_j \rangle >0} \beta_j/2 \rangle \in \mathbb{Z}$.

The imaginary parts of our two hyperplane arrangements, $\Im(H_{F,m})=\{y \in L^{\vee}_{\mathbb{R}} \mid \langle l_F,y \rangle =0\}$ and $\Im(H_{\Gamma,m+n})=\{y \in L^{\vee}_{\mathbb{R}} \mid \langle l_{\Gamma},y \rangle = -\frac{1}{2\pi}\text{log}(|c_{\Gamma}|) \}$, are different in general. However, by \eqref{c}
\[ \text{log}(|c_{\Gamma}|)= \langle l_{\Gamma} , \sum_{j} \text{log}(|n_j|)\beta_j \rangle \]
So if we set  $z:=-\frac{1}{2\pi} \sum_{j} \text{log}(|n_j|)\beta_j$, $\Im(H_{F,m})+z=\Im(H_{\Gamma,m+n})$ and so $H_{F,m}+iz=H_{\Gamma,m+n}$ and we have proved the corollary.

\end{proof}

\end{document}